%% file: whitehead_charvar_paper_archix.tex
\documentclass{article}
\usepackage{amsthm}
\usepackage{amsmath}
\usepackage{color}
\usepackage[all]{xypic}
\usepackage{amssymb}
\usepackage{graphicx}
\usepackage{amsmath}
\usepackage{amsfonts}
\usepackage[all]{xy}

\newtheorem{thmi}{Theorem}
\newtheorem{thm}{Theorem}
\newtheorem{prop}{Proposition}[section]
\newtheorem{lem}{Lemma}[section]
\newtheorem{cor}{Corollary}[section]
\theoremstyle{definition}
\newtheorem{defn}{Definition}[section]

\title{Identifying the Canonical Component for the Whitehead Link }
\author{Emily Landes}

\begin{document}
\maketitle

\begin{abstract}  In this paper we determine topologically the canonical component of the $SL_2(\mathbb{C})$ character
variety of the Whitehead link complement.
\end{abstract}

\begin{section}{Introduction}\label{sec:intro}

\input{asec_intro}

\end{section}

\begin{section}{Preliminaries}\label{sec:preliminary}
Here we briefly describe the $SL_2(\mathbb{C})$ representation and
character varieties.  Standard references on this include \cite{CS1}
and \cite{handbook}.

\begin{subsection}{Representation variety}  For any finitely generated group {\small $\Gamma = \langle g_1,
\dots, g_n\vert r_1,\dots , r_m \rangle $}, the set of
$SL_2(\mathbb{C})$ representations \\ $R(\Gamma)= Hom(\Gamma,
SL_2\mathbb{C})$ has the structure of an affine algebraic set
\cite{CS1}.  We view this space as \\ $R(\Gamma)= \lbrace
(x_1,\dots,x_2)\in (SL_2(\mathbb{C}))^n \vert r_j(x_1,\dots,x_n)=I,
j=1,\dots,m \rbrace$. Notice that $R(\Gamma)$ can be identified with
$r^{-1}(I,\dots,I)$ where $r:(SL_2(\mathbb{C}))^n \to
(SL_2(\mathbb{C}))^m$ is the map $r(x)=(r_1(x),\dots,r_m(x))$. That
$R(\Gamma)$ is an algebraic set follows from the fact that $r$ is a
regular map.  Identifying $SL_2(\mathbb{C})^n$ with a subset of
$\mathbb{C}^{4n}$, we can view $R(\Gamma)$ as an algebraic set over
$\mathbb{C}$.  We should note that the isomorphism class of $R(\Gamma)$ does not depend on
the group presentation and in general, $R(\Gamma)$ is not
irreducible. In fact the abelian representations (i.e.
representations with abelian image) comprise a component of this
algebraic set \cite{CS1}.
\end{subsection}

\begin{subsection}{Character variety}  The character of a representation $\rho:\Gamma \to
SL_2(\mathbb{C})$ is a map $\chi_{\rho}: \Gamma \to \mathbb{C}$
defined by $\chi_{\rho}(\gamma)= tr(\rho(\gamma))$.  The set of
characters is defined to be $X(\Gamma)= \lbrace \chi_{\rho} \vert
\rho \in R(\Gamma)) \rbrace$.  For each $g \in \Gamma$ there is a
regular map $\tau_{g}:R(\Gamma) \to \mathbb{C}$ defined by
$\tau_{g}(\rho)=\chi_{\rho}(g)$.  Let $T$ be the subring of the
coordinate ring on $R(\Gamma)$ generated by $1$ and $\tau_{g}$,
$g\in \Gamma$.  In \cite{CS1} it is shown that the ring T is
finitely generated, for example by $\lbrace
\tau_{g_{i_1}g_{i_2}\dots g_{i_k}} \vert 1 \leq i_1 < i_2 < \dots <
i_k \leq n \rbrace$.  In particular any character $\chi\in
X(\Gamma)$ is determined by its value on finitely many elements of
$\Gamma$. As a result, for $t_1,\dots,t_s$ generators of $T$, the
map $t=(t_1,\dots,t_s):R(\Gamma) \to \mathbb{C}^s$ defined by $\rho
\mapsto (t_1(\rho),\dots,t_s(\rho))$ induces a map $X(\Gamma)\to
\mathbb{C}^s$.  Culler and Shalen use the fact that this map is
injective to show that $X(\Gamma)$ inherits the structure of an
algebraic set (\cite{CS1}).
\end{subsection}

Let $M$ be a hyperbolic manifold and let $\Gamma = \pi_1(M)$.  We
refer to the affine algebraic set $\tilde{X}(\pi_1(M))$ as the affine $SL_2(\mathbb{C})$ character variety of $M$.
The affine canonical component of $\tilde{X}(\pi_1(M))$ is the component containing a character, $\chi_0$, of a discrete
faithful representation and is denoted by $\tilde{X_0}(\pi_1(M))$.  We are interested not so much in the affine model
 but in a closed projective model.  By $SL_2(\mathbb{C})$ character variety and canonical component we
 mean the projective models $X(\pi_1(M))$ and $X_0(\pi_1(M))$ respectively.   For hyperbolic knots and links as
  $\chi_0$ is a smooth point, $X_0$ is unique \cite{Thurston}.  In this context Thurston's Hyperbolic Dehn
  Surgery Theorem states that for an orientable, hyperbolic 3-manifold of finite volume, with $n$-cusps, $X_0$
  has complex dimension $n$.

\vspace{.2in} We will be particularly interested in studying $X_0$
for hyperbolic two component two-bridge link complements.  These
will be complex surfaces.  A more detailed account of the character
variety in this case will be discussed in $\S$ \ref{sec:proof}.

\end{section}

\begin{section}{Algebraic geometry}\label{sec:algebraic geometry}
\input{asec_algebraicgeometry}
\end{section}

\begin{section}{The Whitehead link}\label{sec:proof}
\input{asec_proof_details}

\end{section}

\vspace{.2in}
\begin{section}{Other 2 component 2-bridge link examples}\label{sec:examples}
\input{asec_examples_short}

\end{section}

\bibliographystyle{abbrv}
\bibliography{whbib}

\end{document}

%% file: asec_intro.tex
Since the seminal work of Culler and Shalen, character varieties
have proven to be a powerful tool for studying hyperbolic
3-manifolds; for exampe, they provide efficient means of detecting
essential surfaces in hyperbolic knot complements (\cite{CS1},
\cite{CGLS}, \cite{handbook}).  For such a useful tool, the
character variety itself is a rather basic concept and yet
determining explicit models for even the simplest hyperbolic knot
complements, let alone link complements, is a difficult problem.  We
are particularly interested in the canonical component (i.e that
containing a character of a discrete faithful representation) of
these character varieties.  Only recently have explicit models for
the canonical components of a full family of hyperbolic knots been
determined (\cite{MKR}).

In beginning to understand the explicit models for hyperbolic knots,
our attention extends to constructing canonical components of
hyperbolic 2-component links.   In \cite{MKR} the authors determine
the $SL_2\mathbb{C}$ character varieties for the twist knots.  As
the twist knots can be obtained by Dehn filling one of the cusps of
the Whitehead link complement, we are naturally interested in
constructing the particular character variety of the Whitehead link
and studying the effect Dehn surgery has on the character variety.
The focus of this paper is to construct canonical component for the
$SL_2\mathbb{C}$ character variety of the Whitehead link complement
which we will do as the following theorem

\vspace{.1in}

\begin{thmi}\label{thm:wh} The canonical component of the character variety of the Whitehead link complement
is a rational surface isomorphic to $\mathbb{P}^2$ blown up at $10$
points.\end{thmi} \vspace{.1in}

Understanding an algebro-geometric surface usually means
understanding a surface up to birational equivalence.  One reason
for this is that a lot of the information about a variety is carried
by the birational equivalence class.  For complex curves the
birational equivalence class contains a unique smooth model (up to
isomorphism).  However, for complex surfaces, although there may be
more than one smooth model in a birational equivalence class, there
is a notion of a minimal smooth model. That is the smooth birational
model which has no $(-1)$ curves.

The minimal smooth model for the canonical component of the
Whitehead link is $\mathbb{P}^2$.  We are able to determine this
minimal model from a particular projective model, S, for the
canonical component.  The defining polynomial for the canonical
component of the Whitehead link cuts out an affine surface in
$\mathbb{C}^3$. Compactifying this surface in
$\mathbb{P}^2\times\mathbb{P}^1$ gives the singular surface, $S$.
Although $S$ is a singular surface, it is birational to a conic
bundle which in turn is birational to $\mathbb{P}^2$.

There are three models for the components of the algebraic sets we
discuss in this paper.  There is the affine model defined by an
ideal of polynomials in $\mathbb{C}^3$.  There is the projective
model, which may or may not be smooth, obtained by compactifying the
affine models in $\mathbb{P}^2\times\mathbb{P}^1$.  Finally, if the
projective model is not smooth, there is the smooth projective model
obtained resolving singular points of the projective model.  In this
paper when we refer to components of the character variety we mean
the smooth projective model.  We will specify when speaking of an
affine model or a singular projective model.

Although knowing the birational equivalence class is helpful in
understanding how the variety behaves, determining the variety
topologically requires understanding the isomorphism class.  In this
case, since $S$ is rational, we can use the minimal model to
determine the isomorphism class.  Smooth surfaces birational to
$\mathbb{P}^2$ are isomorphic either to
$\mathbb{P}^1\times\mathbb{P}^1$ or to $\mathbb{P}^2$ blown-up at
$n$ points.  For smooth surfaces, this isomorphism class can be
determined directly from the Euler characteristic.  Although we can
calculate the Euler characteristic for $S$, it does not determine
the isomorphism class since $S$ is not smooth.  Rather than work
with the singular surface $S$ we will work with the smooth surface
obtained by resolving the singularities of $S$.

Away from the four singular points, $S$ looks like a conic bundle in
the sense that it is a bundle over $\mathbb{P}^1$ whose fibers are
conics i.e. curves in $\mathbb{P}^2$ cut out by degree $2$
polynomials.  While the total space of a conic bundle is smooth a
given fiber may not be.  The model, $S$, for the canonical component
of the Whitehead link has six fibers which are not smooth. Five of
these fibers are degenerate i.e. have exactly one singularity
whereas the sixth fiber is a double line i.e. every point is a
singularity.  It is worth remarking that double lines are a fairly
rare feature to conic bundles.  More precisely, all conics can be
parameterized by $\mathbb{P}^5$ and the double lines
correspond to a codimension 3 subvariety (\cite{GH}).  Hence a conic
bundle with a double line fiber corresponds to a line which passes
through a particular codimension 3 subvariety in $\mathbb{P}^5$
which is a rare occurrence.

Since $S$ is birational to a conic bundle and so birational to
$\mathbb{P}^2$, it is isomorphic either to
$\mathbb{P}^1\times\mathbb{P}^1$ or $\mathbb{P}^2$ blown-up at $n$
points.  For $n \leq 8$, the surfaces $\mathbb{P}^2$ blown-up at $n$
points are nice algebro-geometric objets in the sense that they
exhibit only finitely many $(-1)$ curves that is curves with
self-intersection number $-1$.  The canonical component of the
Whitehead link, $\mathbb{P}^2$ blown-up at $10$ points, has
infintely many $(-1)$ curves.

The Whitehead link complement can be obtained by $1/1$ Dehn surgery
on the Borromean rings (the complement of which we will denote by
$M_{br}$).  The manifold which results upon $1/n$ Dehn filling on
one of the cusps of $M_{br}$ is a hyperbolic two component 2-bridge
link complement (\cite{HS}).  For $n=1,\dots,4$ we were able to use
mathematica to determine the polynomials which define the character
varieties of $M_{br}(1/n)$.  For $n=1,\dots,4$, the character
variety of $M_{br}(1/n)$ has a component which is a rational
surface. Aside from the Whitehead link, the rational component(s) of
these character varieties are not canonical components.

\begin{thmi}\label{thm:1/n components} For $n=2,\dots,4$, the character variety of $M_{br}(1/n)$
 has a component which is a rational surface isomorphic to $\mathbb{P}^2$ blown-up at 7 points.
\end{thmi}

Similar to the canonical component of the Whitehead link, all of
these rational surfaces exhibit a double line fiber.  Unlike the
Whitehead link, these rational surfaces have finitely many $(-1)$
curves.   The surface $\mathbb{P}^2$ blown-up at 7 points has
exactly 47 $(-1)$ curves.  There is something to be said for the
fact that examples of this surface come from canonical components
of character varieties of hyperbolic link complements. It is a start
to understanding how topology and algebraic geometry fit together.

We start, in Section \ref{sec:preliminary}, by defining the
character variety.   In Section \ref{sec:algebraic geometry} we
provide some background in algebraic geometry.  The main theorem
will be proved in Section \ref{sec:proof}.  In Section
\ref{sec:examples} we discuss the character varieties for similar
hyperbolic 2 component 2-bridge link complements.

%% file: asec_algebraicgeometry.tex
The purpose of this section is to review the algebro-geometric
concepts relevant to the main proof of the is paper. For more
details see \cite{Hartshorne} or \cite{Shafarevich1} .

\begin{subsection}{Conic bundles} The character varieties of all of our examples have a
component which is a conic bundle.  A conic is a curve defined by a
polynomial over $\mathbb{P}^2$ of degree $2$.  Smooth conics have
the genus zero (\cite{Shafarevich1}) so are spheres.  A degenerate
conic consists of two spheres intersecting one one point. In this
paper the term conic bundle will be used to mean a conic bundle over
$\mathbb{P}^1$ i.e. over a sphere.  Conic bundles are nice
algebro-geometric objects.  Whilst there is no classification of
complex surfaces, there is a classification for the subclass of
$\mathbb{P}^1$ bundles over $\mathbb{P}^1$ which are slightly
different than conic bundles in the sense that conic bundles may can
have fibers with singularities.  Any $\mathbb{P}^1$ bundle over
$\mathbb{P}^1$ comes from a projectivized rank 2 vector bundle over
$\mathbb{P}^1$.  As the rank 2 vector bundles are parametrized by
 $\mathbb{Z}$, the $\mathbb{P}^1$ bundles over $\mathbb{P}^1$ are parameterized by $\mathbb{Z}$.  Each vector bundle
over $\mathbb{P}^1$ can be written as $E=\mathbb{O}\oplus\mathbb{O}(-e)$ (\cite{Beauville}, \cite{Hartshorne}).  Here $\mathbb{O}$ denotes
the trivial rank 2 vector bundle over $\mathbb{P}^1$ and $\mathbb{O}(-e)$ denotes the vector bundle whose section has self-intersection number $e$.

\vspace{.1in}\begin{prop}\label{prop:rational} A conic bundle is a
rational surface. \end{prop}

\vspace{.1in}\begin{proof}  Any conic bundle $T$ can be realized as
a hypersurface defined by a polynomial $f_{T}$ of bidegree $(2,m)$
over $\mathbb{P}^2\times\mathbb{P}^1$.  In particular a generic
fiber of the coordinate projection of $T$ to $\mathbb{P}^1$ is a
nondegenerate conic.  This means that $T$ is locally, and hence
birationally, equivalent to $\mathbb{P}^1\times\mathbb{P}^1$ which
is birationally equivalent to $\mathbb{P}^2$.

Another way to see that $T$ is rational is by looking at the
canonical divisor.  The canonical divisor $K_{T}$ of $T$ is the
canonical divisor $K_{\mathbb{P}^2\times \mathbb{P}^1}$of
$\mathbb{P}^2\times \mathbb{P}^1$ twisted by the divisor class of
$T$, all restricted to $T$.  Namely
$K_{T}=(\mathcal{O}_{\mathbb{P}^2\times\mathbb{P}^1}(-3,-2)\otimes\mathcal{O}_{\mathbb{P}^2\times\mathbb{P}^1}(2,m))\vert
_{S}=\mathcal{O}_{\mathbb{P}^2\times\mathbb{P}^1}(-1,m-2)\vert_{S}$.
In particular, the canonical divisor $K_{T}$ corresponds to the line
bundle
$\mathcal{O}_{\mathbb{P}^2\times\mathbb{P}^1}(-1,m-2)\vert_{S}$ the
number of global sections of which are characterized by the number
of polynomials of bidgree $(-1,m-2)$.  Since there are no
polynomials of bidegree $(-1,m-2)$ there are no global sections on
$T$. The only surfaces in which the canonical bundle has no global
sections are rational and ruled (i.e. birational to $\mathbb{P}^2$
and a fibration over a curve with $\mathbb{P}^2$ fibers).
\end{proof}

\begin{cor}\label{cor:p2n} A conic bundle is isomorphic to either
$\mathbb{P}^1\times\mathbb{P}^1$ or $\mathbb{P}^2$ blown-up at $n$
points for some integer $n$.
\end{cor}

It is a known fact (\cite{Hartshorne} Chapter V) that for two birational varieties the birational
equivalence between them can be written as sequence of blow-ups and blow-downs.  In particular $\mathbb{P}^2$ is
birational to either $\mathbb{P}^1\times\mathbb{P}^1$ or to $\mathbb{P}^2$ blown-up at $n$ points.  Hence any rational
surface is isomorphic to $\mathbb{P}^1\times\mathbb{P}^1$ or to $\mathbb{P}^2$ blown-up at $n$ points.
\end{subsection}

\begin{subsection}{Blow-ups}\label{sec:blow-up}
Blowing-up varieties at points is a standard tool for resolving
singularities and determining isomorphism classes of surfaces and we
make repeated use of such in this paper.

 Since blowing-up is a local process, we can do all of our blow-ups in affine
neighborhoods.  For our purposes, understanding what it means to
blow-up subvarieties of $\mathbb{A}^2$ and $\mathbb{A}^3$ at a point
should be sufficient. For more details refer to \cite{Hartshorne} or
\cite{Shafarevich1}.

Intuitively blowing-up $\mathbb{A}^2$ at a point can be described as replacing a point in $\mathbb{A}^2$ by an exceptional
divisor (i.e. a copy of $\mathbb{P}^1$).  To understand this more concretely, we will describe the blow-up
 of $\mathbb{A}^2$ at the origin.  Consider the product $\mathbb{A}^2\times\mathbb{P}^1$.  Take $x,y$ as the affine coordinates of $\mathbb{A}^2$ and $t,u$ as the
 homogeneous coordinates of $\mathbb{P}^1$.  The blow-up of $\mathbb{A}^2$ at $(0,0)$ is the closed subset
 $Y = \lbrace [x,y:t,u] \vert xu=ty \rbrace$ in $A^2\times\mathbb{P}^1$.  The blow-up comes with a natural map
  $\gamma: Y \to \mathbb{A}^2$ which is just projection onto the first factor.   Notice that the fiber over any point
 $(x,y)\neq(0,0)\in\mathbb{A}^2$ is precisely one point in $Y$.  However, the fiber over $(x,y)=(0,0)$, is a
  $\mathbb{P}^1$ worth of points in $Y$ (i.e. $\lbrace (0,0,t,u) \rbrace \subset Y$ ).
  Since $\mathbb{A}^2 - \lbrace (0,0) \rbrace \simeq Y -
\gamma^{-1}(0,0) $, $\gamma$ is a birational map and
  $\mathbb{A}^2$ is birational to $Y$.  Blowing-up $\mathbb{A}^2$ at a point $p\neq0$ simply amounts to a change in
   coordinates.

  Suppose we want to blow up a subvariety $X \subset \mathbb{A}^2$ at a point, $p$.  Take the blow-up $Y$
  of $\mathbb{A}^2$ at $p$.  Then the blow-up $Bl\vert_p(X)$ of $X$ at $p$ is the closure
  $\overline{\gamma^{-1}(X-p)}$ in $Y$ where $\gamma$ is as described above.
  We note that $Bl\vert_p(X)$ is birational to $X-p$ and if $Bl\vert_p(X)$ is
  smooth, $\gamma^{-1}(p)$ will intersect $Y$ in a zero dimensional variety.

For our paper we need to understand how blowing-up a surface at a
smooth point affects the Euler characteristic.

\begin{prop} \label{prop:blowup} The Euler characteristic of a surface $X$ blown-up at a smooth point $p$ is \\
 $\phantom{.}$ \hspace{1in} $\chi (Bl\vert_p(X)) = \chi(X)+1$  .
\end{prop}

\begin{proof}  To blow-up $X$ at a smooth point $p$ we work locally in an affine neighborhood about $p$.
Near $p$, $X$ is locally $\mathbb{A}^2$ at $0$.   Hence the result
of blowing-up $X$ at $p$ is the same as blowing-up $\mathbb{A}^2$ at
$0$. In terms of the Euler characteristic this amounts to replacing
a point with an exceptional $\mathbb{P}^1$.  In particular
\hspace{.1in} $\chi(Bl_{p}(X)) = \chi(X -\lbrace p\rbrace) +
\chi(\mathbb{P}^1) = \chi(X) +1$ .

\end{proof}

In order to resolve singularities we will need to blow-up
subvarieties of $\mathbb{A}^3$ at a point.  Taking $x_1,x_2,x_3$ as
affine coordinates for $\mathbb{A}^3$ and $y_1,y_2,y_3$ as
projective coordinates for $\mathbb{P}^2$, the blow-up of $A^3$ at
the origin is closed subvariety, $Y'= \lbrace
[x_1,x_2,x_3:y_1,y_2,y_3]\vert x_1y_2=x_2y_1, x_1y_3=x_3y_1,
x_2y_3=x_3y_2 \rbrace$ in $\mathbb{A}^3\times\mathbb{P}^2$.  Just as
in the case of $\mathbb{A}^2$, this blow-up comes with a natural map
  $\gamma: Y' \to \mathbb{A}^3$ which is simply projection onto the first factor.
  Just as before, the fiber over any point
 $(x_1,x_2,x_3)\neq(0,0,0)\in\mathbb{A}^3$ is precisely one point in $Y'$.  However,
 the fiber over $(x_1,x_2,x_3)=(0,0,0)$, is a $\mathbb{P}^2$ worth of points in $Y'$ (i.e. $\lbrace
(0,0,0,y_1,y_2,y_3) \rbrace \subset Y'$ ).  Since $\mathbb{A}^3 -
\lbrace (0,0,0) \rbrace \simeq Y' - \gamma^{-1}(0,0,0) $, $\gamma$
is a birational map and $\mathbb{A}^3$ is birational to $Y'$.
 Blowing-up $\mathbb{A}^3$ at a point $p\neq0$ simply amounts to a
change in coordinates.  To blow up a subvariety $X \subset
\mathbb{A}^3$ at a point, $p$. Take the blow-up $Y'$
  of $\mathbb{A}^3$ at $p$.  Then the blow-up $Bl\vert_p(X)$ of $X$ at $p$ is the closure
  $\overline{\gamma^{-1}(X-p)}$ in $Y'$.  We note that
  $Bl\vert_p(X)$ is birational to $X-p$ and if $Bl\vert_p(X)$ is
  smooth, $\gamma^{-1}(p)$ will intersect $Y'$ in a smooth curve.

In this paper we obtain smooth surfaces by resolving singularities.
As the Euler characteristic of these smooth surfaces helps us
determine the isomorphism class we keep track of how blow-up
singular points affects the Euler characteristic.

\begin{prop} \label{prop:blowup sing} If the blow-up $Bl\vert_p(X)$ of a surface $X$ at a singular point $p$ is smooth,
then the Euler characteristic of $Bl\vert_p(X)$ is
$\chi(Bl\vert_p(X)) = \chi(X)+2g+1$ where $g$ is the genus of the
curve $\gamma^{-1}(p)$ in $Bl\vert_p(X)$.
\end{prop}

\begin{proof}  Away from the point $p$, $X$ is isomorphic to $Bl_p(X) \backslash \gamma^{-1}(p)$.  Hence,
$\chi(Bl_{p}(X)) = \chi(X-p) + \chi(\gamma^{-1}(p))$.  The preimage
$\gamma^{-1}(p)$ in $Bl_p(X)$ is a smooth codimension-1 subvariety
of the fiber over $p$ in $Bl_p(A^3)$.  Since the fiber over $p$ in
$Bl_p(A^3)$ is a $\mathbb{P}^2$, $\gamma^{-1}(p)$ in $Bl_p(X)$ is a
smooth curve of genus $g$.  Hence $\chi(\gamma^{-1}(p)) = 2g+2$ and
$\chi(Bl_{p}(X)) = \chi(X -\lbrace p\rbrace) + \chi(\gamma^{-1}(p))
 =\chi(X)+2g+1$ .

\end{proof}

\end{subsection}

\begin{subsection}{Total transformations}
For the proof of propsition \ref{prop:ec} we will use a total
transform to extend a map $\phi$ between projective varieties. The
description we provide here comes from \cite{Hartshorne} (pg 410).
We begin by setting up some notation.  Let $X$ and $Y$ be projective
varieties.

\begin{defn} A \textit{birational transformation} $T$ from $X$ to $Y$ is an open subset $U\subset X$
and a morphism $\phi: U \to Y$ which induces an isomorphism on the
function fields of X and Y \end{defn}

\noindent Since different maps must agree on the overlap for different open sets, we take the largest
open set $U$ for which there is such a morphism $\phi$.  It is common to say that $T$ is defined at the points of $U$ and

\begin{defn} The \textit{fundamental points} of $T$ are those in the set $X-U$.\end{defn}

\noindent For $G$ the graph of $\phi$ in $U \times Y$, let $\overline{G}$ be the closure of $G$ in $X\times Y$.
Let $\rho_1:\overline{G}\to X$ and $\rho_2:\overline{G}\to Y$ be projections onto the first and second
factors respectively.

\begin{defn}For any subset $Z\subset X$ the \textit{total transform} of $Z$ is $T(Z):=\rho_2(\rho_1^{-1}(Z))$ . \end{defn}
\end{subsection}

\noindent For a point $p \in U$, $T(p)$ is consistent with $\phi(p)$; while for a point $p\in X-U$, $T(p)$ is generally
larger than a single point (in our examples it will be a copy of $\mathbb{P}^1$).

\begin{subsection}{Intersection numbers of curves}
A smooth curve $C$ in $\mathbb{P}^1\times\mathbb{P}^1$ is cut out by
a polynomial $g$ which is homogenous in each of the $\mathbb{P}^1$
coordinates.  We say $g$ has bidgree $(a,b)$ where $a$ is the degree
of $g$ viewed as polynomial over the first factor and $b$ is the
degree of $g$ viewed as a polynomial over the second factor. In the
proof of Theorem \ref{thm:main} we will determine the number of
intersections of two smooth curves in
$\mathbb{P}^1\times\mathbb{P}^1$ based solely on the bidegrees of
their defining polynomials.  Suppose $C_1$ and $C_2$ are two smooth
curves cut out by irreducible polynomials $g_1$ and $g_2$ of
bidegrees $(a_1,b_1)$ and $(a_2,b_2)$ respectively.  Counting
 multiplicities, $C_1$ and $C_2$ intersect in $a_1b_2+a_2b_1$ points \cite{Hartshorne} 5.1).
\end{subsection}

\begin{subsection}{Geometric genus} The geometric genus, $p_g$, of a projective variety, $S$, is the dimension of the
vector space of global sections $\Gamma(X,\omega_k)$ of the
canonical divisor $w_k$.  For a complex curve, the geometric genus
coincides with the topological genus and can thus be used to
topologically determine the character varieties of hyperbolic knot
complements.  Unfortunately for complex surfaces, the geometric
genus does not carry as direct topological information (for instance
it appears as $h^{2,0}$ in the Hodge decomposition \cite{GH}).
However, as it may still be helpful in determining which varieties
can arise as the character varieties of hyperbolic two component
link complements, it it worth keeping track of this value. For a
hypersuface, $Z$, in $\mathbb{P}^2\times\mathbb{P}^1$ defined by a
polynomial $f$ of bidegree $(a,b)$ the geometric genus is $p_g(Z)=
\frac{(a-1)(a-2)(b-1)}{2}$.

We give a brief description of this here. As the group of linear
equivalence classes of divisor of $\mathbb{P}^2\times\mathbb{P}^1$
is $Pic(\mathbb{P}^2\times\mathbb{P}^1) \cong
\mathbb{Z}\times\mathbb{Z}$, we think of the divisors of
$\mathbb{P}^2\times\mathbb{P}^1$ as elements of
$\mathbb{Z}\times\mathbb{Z}$.  For a linear class with
representative divisor $D$ on $\mathbb{P}^2\times\mathbb{P}^1$,
 there is an associated vector space, $L(D)$ of principal divisors $E$ such that D+E is effective.  The vector space
 $L(D)$ is in one-to-one correspondence with the vector space of global sections of the line bundle $\mathbb{L}(D)$
 on $\mathbb{P}^2\times\mathbb{P}^1$.  As the vector space of global sections of $\mathbb{L}(D)$ corresponds to the
 space of polynomials over $\mathbb{P}^2\times\mathbb{P}^1$ with the same bidegree as that which cuts out $D$, the
 restrictions of these polynomials to $S$ which are nonzero on $S$, correspond to the vector
space of global sections of $D$ on $S$.  That is to say the kernel of the surjective map
$L(D)\twoheadrightarrow \L(D)\vert_{S}$ is those
polynomials which vanish on $S$.   When $D$ is the the canonical divisor $K_S$ of the surface $S$, assuming
 all the restricted polynomials are nonzero on $S$, the geometric genus of the surface $g_g(S)$ is then
just the dimension of the vector space of these polynomials.

For the hypersurface $S$ defined by $f$, we can use the adjunction
formula to determine $K_S$. Namely, $K_S =
[K_{\mathbb{P}^2\times\mathbb{P}^1}\otimes \mathbb{O}(S)]\vert_{S}$.
The canonical divisor $K_{\mathbb{P}^2\times\mathbb{P}^1}$ of
$\mathbb{P}^2\times\mathbb{P}^1$ is
$(-3,-2)\in\mathbb{Z}\times\mathbb{Z}$. and the divisor class
$\mathbb{O}(S)$ = $(a,b)\in
Pic(\mathbb{P}^2\times\mathbb{P}^1)\cong\mathbb{Z}\times\mathbb{Z}$
since $f$ has bidgeree $(a,b)$.  Hence, $K_S=(a-3,b-2)$.  Since the
linear class of divisors of $K_S=(a-3,b-2)$ corresponds to a
polynomials of bidgree $(a-3,b-2)$, the global sections of line
bundle associated to $K_S=(a-3,b-2)\vert_{S}$ correspond to
polynomials of bidgree $(a-3,b-2)$.  Since $S$ is a hypersuface
defined by the irreducible polynomial $f$, no polynomial of bidgeree
$(a-3,b-2)$ can vanish on all of $S$.  Hence, the geometric genus of
the surface $g_g(S)$ is then just the dimension of the vector space
of polynomials over $\mathbb{P}^2\times\mathbb{P}^1$ of bidegree
$(a,b)$. Determining this dimension is a matter of counting
monomials of bidegree $(a-3,b-3)$ for which there are $
\frac{(a-1)(a-2)(b-1)}{2}$.

\end{subsection}

\begin{subsection}{Projective models for character varieties}
The affine varieties with which we are concerned are all
hypersurfaces in $\mathbb{C}^3$ i.e. they are zero sets
$Z(\tilde{f})$ of a single smooth polynomial $\tilde{f}\in
\mathbb{C}[x,y,z]$.  Finding the right projective completion is
tricky, especially with complex surfaces since different projective
completions may result in non-isomorphic models.   It might seem
natural to take projective closures in $\mathbb{P}^3$.  One problem
with compactifying in $\mathbb{P}^3$ is that, generally, this
projective model has singularities which take more than one blow-up
to resolve. Following the work of \cite{MKR} it is more natural to
consider the compactification in $\mathbb{P}^2\times\mathbb{P}^1$.
This compactification does result in a singular surface.  However,
the singularities are manageable and away from the singularities
this model has the nice structure of a conic bundle.  Hence, for
these reasons, this is the projective model we choose to use for our
examples.

Given an affine variety $Z(\tilde{f})$ defined by a polynomial
$\tilde{f}\in \mathbb{C}[x,y,z]$, we construct the projective
closure by homogenizing $\tilde{f}$.  Let $a$ be the degree of
$\tilde{f}$ when viewed as a polynomial in variables $x$ and $y$.
Let $b$ be the degree of $\tilde{f}$ when viewed  as a polynomial in
the variable $z$.  The projective model in
$\mathbb{P}^2\times\mathbb{P}^1$ of the affine variety
$Z(\tilde{f})$ is cut out by the homogenous polynomial
 $f=u^{a}w^{b}\tilde{f}(\frac{x}{u},\frac{y}{u},\frac{z}{w})$ where $x,y,u$ are $\mathbb{P}^2$
 coordinates and $z,w$ are $\mathbb{P}^1$ coordinates.  Notice that every monomial which appears
 in $f$ has degree $a$ in the $\mathbb{P}^2$ coordinates and degree $b$ in the $\mathbb{P}^1$
 coordinates so $f$ has bidgereee $(a,b)$.
\end{subsection}

%% file: asec_proof_details.tex
Let $W$ denote the complement of the Whithead link in $S^3$ and let
$\Gamma_W = \pi_1(W)$. Then $\Gamma_W = \langle a,b| aw=wa \rangle$
where $w$ is the word $w=bab^{-1}a^{-1}b^{-1}ab$.

\begin{figure}[!hbp]
\begin{center}
\includegraphics[scale=.3]{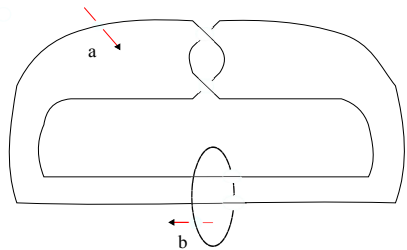}
\caption{Whitehead link \label{fig:whgenerators}}
\end{center}
\end{figure}

\begin{prop} $\tilde{X}(W)$ is a hypersurface in $\mathbb{C}^3$.
\end{prop}

\begin{proof} To determine the defining polynomial for $\tilde{X}(W)$ in $\mathbb{C}^3$ we look at the image of $R(W)$ under the map $t=(t_1,\dots,t_s): R(W)\to \mathbb{C}^s$ as defined in $\S$ \ref{sec:preliminary}.  We begin by establishing the defining ideal for $R(W)$.
\vspace{.1in} \noindent Any representation of $\rho \in R(W)$ can be
conjugated so that \vspace{.1in}
\begin{displaymath}
\bar{a}=\rho(a)= \left(
\begin{array}{ll}m&1 \\ 0 & m^{-1}\end{array}\right ) \qquad \bar{b}=\rho(b)=  \left(\begin{array}{ll}s&0\\r&s^{-1}\end{array}\right)
\end{displaymath}

\vspace{.1in}
\noindent The polynomials which define $R(W)$ then come from the relation $\bar{w}\bar{a}-\bar{a}\bar{w} = 0$.
Writing {\small $\rho(w)=\left(\begin{array}{cc} w_{11} & w_{12} \\ w_{21} & w_{22}
\end{array}\right)$}, we see that $$\bar{w}\bar{a}-\bar{a}\bar{w} = \left(\begin{array}{cc} -w_{21} & w_{11} + w_{12}(m^{-1}-m)-w_{22} \\ w_{21}(m-m^{-1}) & w_{21}  \end{array} \right) $$

\noindent Hence, the representation variety is cut out by the
ideal $\langle p_1, p_2  \rangle \subset \mathbb{C}[m,m^{-1}r,s,s^{-1}]$ where $p_1= w_{21}$ and $p_2= w_{11}
+ w_{12}(m^{-1}-m)-w_{22} $ .

\vspace{.1in}
\noindent For the Whitehead link

\begin{displaymath}
\begin{array}{lcl}
p_1 & = & m^{-2}s^{-2}r(r - m^2 r + m s - m^3 s + 2 m r^2 s -
m^3r^2s - r s^2 \nonumber \\
& & + 4 m^2 r s^2 - m^4 r s^2 + m^2 r^3 s^2 - ms^3 \nonumber \\
 & & + m^3s^3 - m r^2 s^3 + 2 m^3 r^2 s^3 - m^2 r s^4 + m^4 r s^4) \\
& & \nonumber \\
p_2 & = & m^{-2}s^{-3}(-1 + s)(1 + s)(r - m^2 r + m s - m^3 s +
2mr^2s \nonumber \\
& &  - m^3 r^2 s - r s^2 + 4 m^2 r s^2 - m^4 r s^2 + m^2 r^3 s^2
\nonumber\\
& & - m s^3 + m^3 s^3 - m r^2 s^3 + 2 m^3 r^2 s^3 - m^2 r s^4 + m^4r
s^4)
\end{array}
\end{displaymath}
\vspace{.1in}

\noindent Neither $p_1$ nor $p_2$ are irreducible.  In fact their $GCD$ is nontrivial.  Let $p=GCD(p_1,p_2)$.
That is
\begin{displaymath}
\begin{array}{lcl}
p & = & m^{2}s^{3}(r - m^2 r + m s - m^3 s + 2 m r^2 s - m^3 r^2 s
\nonumber \\
& & - r s^2 + 4m^2r s^2 - m^4rs^2 + m^2r^3s^2 - m s^3 \nonumber \\
& & + m^3 s^3 - m r^2 s^3 + 2 m^3 r^2 s^3 - m^2 r s^4 + m^4 r s^4) \\
\end{array}
\end{displaymath}
\vspace{.1in}

Setting $g_1 = \frac{p_1}{p} = rs$ and $g_2=\frac{p_2}{p}= s^2-1$,
we can view the representation variety as $Z(\langle g_1p, g_2p
\rangle) = Z(\langle g_1,g_2 \rangle) \cup Z(\langle p \rangle)$.
The ideal $\langle g_1,g_2 \rangle$ defines the affine variety
$R_a=\lbrace (a,1/a,\pm 1, \pm 1,0) \rbrace \subset \mathbb{C}^6$
which is just two copies of $\mathbb{A}^1$. The variety $R_a$ is
precisely all the abelian representations of $R(W)$.  With the $r$
coordinate zero and the $s$ coordinate $\pm 1$, any representation
in $R_a$ sends $b$ to $\pm I$.  We are interested the components of
the representation variety which contain discrete faithful
representations.  All of these representations are in the subvariety
$R = Z(\langle p \rangle)$ of $R(W)$.   Hence we are concerned with
with the image of $R$ under $t$ in the character variety.

\vspace{.2in} The coordinate ring $T_w$ for the Whitehead link
character variety is generated by the trace maps
\begin{displaymath}
\lbrace \tau_{a},\tau_{b},\tau_{ab} \rbrace \end{displaymath} With
these generators the map $t=(\tau_{a},\tau_{b},\tau_{ab}):R \to
\mathbb{C}^3$ is $t(\rho)=(m+m^{-1},s+s^{-1}, ms+m^{-1}s^{-1}
+r)=(x,y,z)$.  Let $X'$ denote the image of $R$ under $t$.  Then the
map $t:R \to X'$ induces an injective map,
$t^*:\mathbb{C}[X']\to\mathbb{C}[R]$ on the coordinates rings of
$X'$ and $R$.   The coordinate ring of R is \\ $\mathbb{C}[R] =
\mathbb{C}[m,m^{-1},s,s^{-1},r]/<p>$ so the image of
$\mathbb{C}[X']$ under $t^*$ is
\begin{displaymath}
\mathbb{C}[m,m^{-1},s,s^{-1},r]/<p,x=m+m^{-1},y=s+s^{-1},z=
ms+m^{-1}s^{-1} +r>
\end{displaymath}
 which is isomorphic to $\mathbb{C}[x,y,z]/<\tilde{f}>$ where $\tilde{f}=-xy - 2z + x^2 z +
y^2 z - xy z^2 + z^3 $ .  Since $\tilde{f}$ is smooth,  $X'$ is the
affine variety $Z(\tilde{f})$.  Now $X'$ is a smooth affine surface
in $\mathbb{C}^3$ containing the surface $\tilde{X}_0$.  Hence $X'
=\tilde{X}_0$ and so $\tilde{X}_0$ is the hypersurface
$Z(\tilde{f})$.
\end{proof}

\vspace{.2in}  Throughout the rest of this section we will denote
compact model $X_0(W)$ for the canonical component of the Whitehead
link by $S$.  We use the compact model obtained by taking the
projective closure in $\mathbb{P}^2 \times \mathbb{P}^1$. With $x,
y, u$ the $\mathbb{P}^2$ coordinates and $z,w$ the $\mathbb{P}^1$
coordinates, this compactification for the canonical component $S$
is defined by $f=-w^3 x y - 2 u^2 w^2 z + w^2 x^2 z + w^2 y^2 z - w
x y z^2 + u^2 z^3$.  This surface, $S$ is not smooth. It has
singularities at the four points:
\begin{center}
\vspace{.1in}
 $
\begin{array}[!hbp]{lcl}
 s1 & = & [1,0,0,1,0]  \\ \\
 s2 & = & [0,1,0,1,0]  \\ \\
 s3 & = & [1,-1,0,1,-1]  \\ \\
 s4 & = & [1,1,0,1,1]
\end{array}
$ \vspace{.1in}
\end{center}
Our goal is to determine topologically, the smooth surface
$\tilde{S}$ obtained by resolving the singularities of $S$ . We do
this in the following theorem.

\vspace{.2in}\begin{thm}\label{thm:main} The surface $\tilde{S}$ is
a rational surface isomorphic to $\mathbb{P}^2$ blown-up at $10$
points.\end{thm}

\noindent The Euler characteristic of $\tilde{S}$ together with the
fact that $\tilde{S}$ is rational is enough to determine $\tilde{S}$
up to isomorphism.

\begin{lem}\label{lem:birationalconicbundle} $\tilde{S}$ is birational to a conic bundle. \end{lem}

\begin{proof}
Consider the projection $\pi_{\mathbb{P}^1}:S\to \mathbb{P}^1$. The
fiber over $[z_0,w_0]\in \mathbb{P}^1$ is the set
of points  $[x,y,u:z_0,w_0]$ which satisfy \\
\vspace{.05in}
$ -w_{0}^3x y - 2 u^2 w_{0}^2 z_{0} + w_{0}^2 x^2 z_{0} + w_{0}^2 y^2 _{0}z - w_{0} x y z_{0}^2 + u^2 z_{0}^3=0 $ \\
\vspace{.05in} This is the zero set of a degree $2$ polynomial in
$\mathbb{P}^2$ which is a conic.  Away from the four singularities,
$S$ is isomorphic to a conic bundle.  Hence, $S$ is birational to a
conic bundle.  Since $\tilde{S}$ is obtained from $S$ by a series of
blow-ups, $\tilde{S}$ is birational to $S$ and so birational to a
conic bundle.
\end{proof}

\newpage

\vspace{.2in} Applying Proposition \ref{prop:rational} we now have
that $\tilde{S}$ is rational surface. Since $S$ has degenerate
fibers, $\tilde{S}$ is not isomorphic to
$\mathbb{P}^1\times\mathbb{P}^1$ (see figure \ref{pic:whconic}).
\begin{figure}[!hbp]
\begin{center}
\includegraphics[scale=.5]{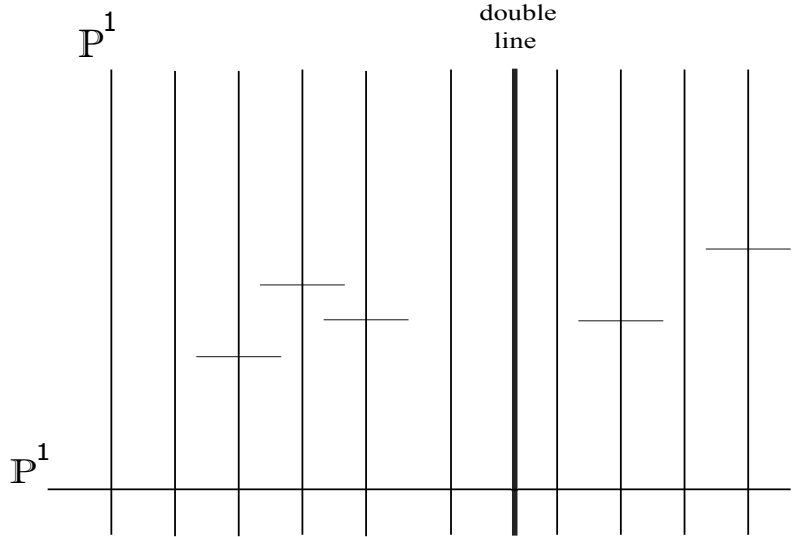}
\caption{Canonical component of the Whitehead
link.\label{pic:whconic} }
\end{center}
\end{figure}
\vspace{.15in} \noindent So, by corollary \ref{cor:p2n}  $\tilde{S}$
is topologically $\mathbb{P}^2$ blown-up at $n$ points.  It follows
from proposition \ref{prop:blowup} that $\chi(\tilde{S}) =
\chi(\mathbb{P}^2) + n = 3+n$.  Thus we can determine $n$ from the
Euler characteristic of $\tilde{S}$ .

To calculate the Euler characteristic of $\tilde{S}$ we use the
Euler characteristic of $S$.  Since the smooth surface $\tilde{S}$
is obtained from $S$ by a series of blow-ups, we can use proposition
\ref{prop:blowup sing} to write $\chi(\tilde{S})$ in terms of
$\chi(S)$.

\vspace{.2in}\begin{lem}\label{lem:blowup and chi}
$\chi(\tilde{S})=\chi(S) + 4$
\end{lem}

\begin{proof}
%this sucks

The smooth surface, $\tilde{S}$, is obtained by resolving the four
singularities, $s_i$, of $S$ listed above.  Above the singularities,
a local model for $\tilde{S}$ can be obtained by blowing-up $S$ in
an affine neighborhood of each of the singular points.   Away from
the singularities we can take the local model for $S$ as a local
model for $\tilde{S}$ since $S$ and $\tilde{S}$ are locally
isomorphic there.  Each of the singularities is nice in the sense
that it takes only one blow-up to resolve them. Hence, in terms of
the Euler characteristic, we haves
\begin{equation}\chi(\tilde{S})=\chi(S-{\lbrace s_i\rbrace})
+\Sigma_{i=1}^{4}\chi(\tilde{s_i}) \end{equation} \noindent where
for $i=1\dots 4$, $\tilde{s_i}$ denotes the preimage of $s_i$ in
$\tilde{S}$.  Determining the Euler characteristic of $\tilde{S}$ in
terms of that for $S$ reduces to determining $\tilde{s_i}$.

To blow-up $S$ at $s_1 = [1,0,0,1,0]$ we consider the affine open
set $A'_1$ where $x\neq 0$ and $z \neq 0$.  Noticing that the
singularities $s_3$ and $s_4$ are in $A'_1$, we look at the blow-up
of $S$ at $s_1$ in the affine open set
$A_1=A'_1\setminus\{s_3,s_4\}$. Local affine coordinates for $A_1
\cong \mathbb{A}^3$ are $y,u,w$.  So to blow-up $S$ at $s_1$ we
blow-up $X_1=Z(f\vert_{x=1,z=1})$ at $[y,u,w]=[0,0,0]$ in $A_1$.  As
described in section  \ref{sec:blow-up}, the blow-up of $X_1$ at
$[0,0,0]$ is the closure of the preimage of $X_1 - [0,0,0]$ in
$Bl\vert_{[0,0,0]}(A_1)$.  Using coordinates $a, b, c$ for
$\mathbb{P}^2$, the blow-up $Y_1$ of $X_1$ at $[0,0,0]$ is the
closed subset in $A_1\times\mathbb{P}^2$ defined by the equations
\begin{eqnarray}
f_1 = f\vert_{x=1,z=1} & =& u^2 + w^2 - 2 u^2 w^2 - w y - w^3y + w^2y^2 \\
e_1 &= &y*b - u*a \\
e_2 &=& y*c - w*a \\
 e_3 &=& u*c - w*b
\end{eqnarray}
\noindent We determine the local model above $s_1$ and check for
smoothness by looking at $Y_1$ in the affine open sets define by
$a\neq 0$, $b\neq 0$, and $c\neq 0$.

First we look at $Y_1$ in the affine open set defined by $a\neq0$
(i.e. we can set $a=1$).  In this open set the defining equations
for $Y_1$ become

\begin{eqnarray}
f_1 & =& u^2 + w^2 - 2 u^2 w^2 - w y - w^3y + w^2y^2 \\
e_1 &= &y*b - u \\
e_2 &=& y*c - w \\
e_3 &=& u*c - w*b
\end{eqnarray}

\noindent Using equations $e1$ and $e2$ and substituting for $u$ and
$w$ in $f_1$ we obtain the local model, $y^2(-b^2 + c - c^2 - c^2
y^2 + 2 b^2 c^2 y^2 + c^3 y^2)$.  The first factor is the
exceptional plane, $E_1$ and the other factor is the local model for
$Y_1$. Notice that $E_1$ and $Y_1$ meet in the smooth conic $-b^2 +
c - c^2$.  So, in this affine open set, the local model above the
singularity $s_1$ is a conic; a $\mathbb{P}^1$.  Since the only
places all the partial derivatives of the second factor vanish are
over the singular points $s_3$ and $s_4$, this model is smooth in
$A_1\times\mathbb{P}^2$.

Next we look at $Y_1$ in the affine open set defined by $b\neq0$. In
this open set the defining equations for $Y_1$ become
\begin{eqnarray}
f_1& =& u^2 + w^2 - 2 u^2 w^2 - w y - w^3y + w^2y^2 \\
e_1 &= &y - u*a \\
e_2 &=& y*c - w*a \\
 e_3 &=& u*c - w
\end{eqnarray}

\noindent Substituting into $f_1$, we obtain the local model, $u^2(1
- a c + c^2 - 2 c^2 u^2 + a^2 c^2 u^2 - a c^3 u^2)$.  Again, the
first factor is the exceptional plane, $E_1$ and the other factor is
the local model for $Y_1$. Notice that $E_1$ and $Y_1$ meet in the
smooth conic $1 - a c + c^2$. So, in this affine open set, the local
model above the singularity $s_1$ is a conic. Since all the partial
derivatives of the second factor do not simultaneously vanish, this
model is smooth in $A_1\times\mathbb{P}^2$.

Finally we look at $Y_1$ in the affine open set defined by $c\neq0$.
In this open set the defining equations for $Y_1$ become

\begin{eqnarray}
f_1 & =& u^2 + w^2 - 2 u^2 w^2 - w y - w^3y + w^2y^2 \\
e_1 &= &y*b - u*a \\
e_2 &=& y - w*a \\
 e_3 &=& u - w*b
\end{eqnarray}

\noindent Substituting into $f_1$, we obtain the local model, $w^2(1
- a + b^2 - a w^2 + a^2 w^2 - 2 b^2 w^2)$.  The first factor is the
exceptional plane, $E_1$ and the other factor is the local model for
$Y_1$. Notice that $E_1$ and $Y_1$ meet in the smooth conic $1 - a +
b^2$. So, in this affine open set, the local model above the
singularity $s_1$ is a conic.  Since the only places all the partial
derivatives of the second factor vanish simultaneously are $s_3$ and
$s_4$, this model is smooth in $A_1\times\mathbb{P}^2$.

Rehomogenizing we see that blowing-up yields a smooth local model
which intersects the exceptional plane above $s_1$ in the conic
defined by $c^2 - a + b^2$.  Hence $\chi(\tilde{s_1})=2$.

Blowing-up $S$ at $s_2$, $s_3$ and $s_4$ is similar to blowing-up
$S$ at $s_1$.  For detailed calculations, we refer the reader to the
Thesis (what's the proper way to reference this?).

In each case local model for $Bl(S)\vert_{s_i}$ intersects the
exceptional plane above $s_i$ in a smooth conic. Hence
$\chi(\tilde{s_i})=2$ for $i=1,\dots,4$ and

\begin{eqnarray}
\chi(\tilde{S}) & = & \chi(S-{\lbrace s_i\rbrace})
+\Sigma_{i=1}^{4}\chi(\tilde{s_i}) \\
&=& \chi(S) - \Sigma_{i=1}^4\chi(s_i)
+\Sigma_{i=1}^{4}\chi(\tilde{s_i}) \\
& =& \chi(S) - 4 + 4(2) \\
& = &\chi(S) + 4
\end{eqnarray}
\end{proof}

\vspace{.2in}\begin{prop}\label{prop:ec} The Euler characteristic of
the surface $S$ is $\chi(S) = 9$.\end{prop}

\vspace{.2in} \noindent To calculate the Euler characteristic we
will appeal to the map $\phi: S \to \mathbb{P}^1 \times
\mathbb{P}^1$ defined by $[x,y,u:z,w] \to [x,y:z,w]$ on a dense open
set of S. That the map $\phi$ is generically 2-to-1 makes it an
attractive tool in determining the Euler characteristic of $S$.
However, in order to calculate the Euler characteristic of $S$ we
must understand the map $\phi$ everywhere not just generically.  To
this affect, there are four aspects we need to consider.  The map
$\phi$ is neither surjective nor defined at the three points
$P=\lbrace (0,0,1,0,1), (0,0,1,1,\pm \frac{1}{\sqrt{2}}) \rbrace$.
Over six points in the $\mathbb{P}^1 \times \mathbb{P}^1$ the fiber
is a copy of $\mathbb{P}^1$.  Finally, the map is branched over
three copies of $\mathbb{P}^1$.  We explain how to alter the Euler
characteristic calculation to account for each of these situations.

\vspace{.2in} \begin{lem}  The image of $\phi$ on $U=S-P$ is
$\mathbb{P}^1\times\mathbb{P}^1 - Q $ where
$$\begin{array}{lcll} Q
&  & = &  \mathbb{P}^1\times \lbrace [0,1] \rbrace \diagdown \lbrace
[1,0,0,1], [0,1,0,1]  \rbrace  \\
& & & \\
&& \hspace{.1in} \cup &  \mathbb{P}^1\times\lbrace [1,
\frac{1}{\sqrt{2}}] \rbrace \diagdown \lbrace [\frac{1}{\sqrt{2}},
1,1,\frac{1}{\sqrt{2}}],[\sqrt{2},
1,1,\frac{1}{\sqrt{2}}] \rbrace   \\
&&& \\
&& \hspace{.1in} \cup &
\mathbb{P}^1\times\lbrace[1,-\frac{1}{\sqrt{2}}] \rbrace \diagdown
\lbrace [-\frac{1}{\sqrt{2}}, 1,1,-\frac{1}{\sqrt{2}}],[-\sqrt{2},
1,1,-\frac{1}{\sqrt{2}}] \rbrace \\
\end{array} $$
\end{lem}

\vspace{.1in}\begin{proof} We can see that this is in fact the image
by viewing $f$ as a polynomial in $u$ with coefficients in
$\mathbb{C}[x,y,z,w]$. Namely $f= g +u^2h$ where $g=
-w^3xy+w^2x^2z+w^2y^2z-wxyz^2$ and $h = z(z^2 - 2w^2)$.  The image
of $\phi$ is the collection of all points $[x,y,z,w]\in
\mathbb{P}^1\times\mathbb{P}^1$  except those for which $f(x,y,z,w)
\in \mathbb{C}[u]$ is a nonzero constant.  The polynomial
$f(x,y,z,w)$ is a nonzero constant whenever $h=0$ and $g\neq 0$.  It
is easy to see that $h=0$ whenever $[z,w]=\lbrace
[0,1],[1,\pm\frac{1}{\sqrt{2}}] \rbrace$.  For each of the $z,w$
coordinates which satisfy $h$, there are two $x,y$ coordinates which
satisfy $g(z,w)$.  Hence the image of $\phi$ on $U$ is all of
$\mathbb{P}^1\times\mathbb{P}^1$ less the three twice punctured
spheres as listed above.

\end{proof}

\vspace{.2in}\begin{lem}  The map $\phi$ smoothly extends to all of
$S$.\end{lem}

\vspace{.1in}\begin{proof}We can extend the map $\phi$ to all of $S$
by using a total transformation. Let $U=S-P$. Then $U$ is the
largest open set in $S$ on which $\phi$ is defined. Let
$\overline{G(\phi,U)}$ be the closure of the graph of $\phi$ on $U$.
We can then smoothly extend the map $\phi$ to all of $S$ by
defining$\phi$ at each $p_i\in P$ to be
$\phi(p_i):=\rho_2\rho_1^{-1}(p_i)$ where $\rho_1: \overline{G}\to
S$ and $\rho_1:\overline{G} \to \mathbb{P}^1\times\mathbb{P}^1$ are
the natural projections.  Note, that the for $s\in U$,
$\rho_2\rho_1^{-1}(s)$ coincides with the original map so that this
extension makes sense on all of $S$.  Now, the closure of the graph
is $\overline{G} =\lbrace[x,y,u,z,w:a,b,c,d]\vert f=0,
ay=bx,cw=dz\rbrace $. So, $\phi$ extends to $S$ as follows:
\vspace{.1in}

$
\begin{array}{lcl}
\phi( (0,0,1,0,1)) & =  & \lbrace [a,b,0,1] \rbrace  \\
\\
\phi( (0,0,1,1,\frac{1}{\sqrt{2}}) )& = &\lbrace
(a,b,1,\frac{1}{\sqrt{2}}) \rbrace \\ \\

\phi((0,0,1,1,-\frac{1}{\sqrt{2}}) )& = &\lbrace
(a,b,1,-\frac{1}{\sqrt{2}}) \rbrace
\end{array}
$ \vspace{.2in}

\end{proof}

\noindent  Notice that the set
$Q\subset\mathbb{P}^1\times\mathbb{P}^1$, which is not contained in
the image of $\phi$ on $U$, is contained in the image of $\phi$ on
$P$. That the extension $\phi$ maps three points in $S$ to not just
three disjoint $\mathbb{P}^1$'s in $\mathbb{P}^1\times\mathbb{P}^1$
but to the three disjoint $\mathbb{P}^1$'s which are are missing
from the image of $\phi$ on $U$ will be important for the Euler
characteristic calculation.

\vspace{.2in}\begin{lem} There are six points in
$\mathbb{P}^2\times\mathbb{P}^1$, the collection of which we will
call $L$, whose fiber in $S$ is infinite.\end{lem}

\vspace{.1in}\begin{proof} Thinking of $f$ as a polynomial in the
variable $u$ with coefficients in $\mathbb{C}[x,y,z,w]$, we see that
the points in $\mathbb{P}^1\times\mathbb{P}^1$ which are
simultaneously zeros of these coefficient polynomials are precisely
the points in $\mathbb{P}^1\times \mathbb{P}^1$ whose fiber is
infinite.  We note here that the points of $L$ are precisely the
punctures of the three punctured spheres which are not in the image
of $\phi\vert_U$. The preimage of $L$ in S is the union of six
$\mathbb{P}^1$'s each intersecting exactly one other $\mathbb{P}^1$
in one point. These three points of intersection are the points on
the $\mathbb{P}^1$'s where the coordinate $u$ goes to infinity which
is equivalent to the points where the $x$ and $y$ coordinates go to
zero.  Thus these intersection points are precisely the points in P.
The points in $L$ along with their infinite fibers in
$\mathbb{P}^1\times\mathbb{P}^1$ are listed below.

\vspace{.1in} $
\begin{array}{lcl}
 [1,0,0,1] & \textrm{has fiber} & \lbrace[1,0,u,0,1] \rbrace \supset [0,0,1,0,1]
 \\ \\

 [0,1,0,1] & \textrm{has fiber} & \lbrace[0,1,u,0,1] \rbrace \supset
 [0,0,1,0,1]\\ \\

 [1,\sqrt{2},1,\frac{1}{\sqrt{2}}]& \textrm{has fiber}& \lbrace [1,\sqrt{2},u,1,\frac{1}{\sqrt{2}}]\rbrace \supset [0,0,1,1,\frac{1}{\sqrt{2}} ]
 \\ \\

 [1,\frac{1}{\sqrt{2}},1,\frac{1}{\sqrt{2}}]& \textrm{has fiber}& \lbrace [1,\frac{1}{\sqrt{2}},u,1,\frac{1}{\sqrt{2}}]   \rbrace \subset [0,0,1,1,\frac{1}{\sqrt{2}} ]
 \\ \\

 [1,-\sqrt{2},1,-\frac{1}{\sqrt{2}}]& \textrm{has fiber}& \lbrace  [1,-\sqrt{2},u,1,-\frac{1}{\sqrt{2}}] \rbrace \supset [0,0,1,1,-\frac{1}{\sqrt{2}}]
 \\ \\

 [1,-\frac{1}{\sqrt{2}},1,-\frac{1}{\sqrt{2}}] & \textrm{has fiber}& \lbrace [1,-\frac{1}{\sqrt{2}},u,1,-\frac{1}{\sqrt{2}}] \rbrace
 \supset
 [0,0,1,1,-\frac{1}{\sqrt{2}}]

\end{array}
$ \vspace{.1in}

\end{proof}

\noindent In calculating the Euler characteristic we will use the
fact that the preimage of $L$ in $S$ are six $\mathbb{P}^1$'s which
intersect in pairs at ideal points in the set $P\subset S$. In fact,
each point in $P$ appears as the intersection of two of these fibers
and the image of $P$ under $\phi$ is precisely $L$.

\vspace{.2in}  Let $B$ denote the branch set of $\phi$ in
$\mathbb{P}^1\times\mathbb{P}^1$.  We have the following lemma.

\vspace{.2in}\begin{lem} $\chi(B) =2$. \end{lem}

\vspace{.1in}\begin{proof} The branch set, or at least the places
where $\phi$ is not one-to-one, consists of the points in $S$ which
also satisfy the coordinate equation $u=0$.  The image, $B \subset
\mathbb{P}^1\times\mathbb{P}^1$, of this branched set, is the union
of the three varieties, $B_1$, $ B_2$, and $B_3$ defined by the
respective three polynomials $f_1=wy-xz$, $f_2=wx-yz$, and $f_3=w $
which are all $\mathbb{P}^1$ 's.  From the bidegrees of the $f_i$ we
know that $B_3$ intersects each of $B_1$ and $B_2$ in one point
($[0,1,1,0]$ and $[ 1,0,1,0]$ respectively) while $B_1$ and $B_2$
intersect in two points ( $[1,-1,-1,1]$ and$[1,1,1,1]$ ). Again
thinking of $f$ as a polynomial in $u$ we can write $f$ as $f=A +
u^2B$ where $A$ and $B$ are polynomials in $\mathbb{C}[x,y,z,w]$.
Since $L$ cut out by the ideal $<A,B>$ and B is cut out by the ideal
$<A>$, $L$ is a subvariety of $B$.  That each of six points in
$\mathbb{P}^1 \times \mathbb{P}^1$ whose fiber is infinite is also a
branched point is necessary for the Euler characteristic
calculation.

\end{proof}

\vspace{.2in} Now that we understand the map $\phi$ everywhere we
can calculate the Euler characteristic of $S$; prove Proposition
\ref{prop:ec}.

\vspace{.1in}\begin{proof}{(Proposition \ref{prop:ec})} Since the
set of points in $\mathbb{P}^1\times\mathbb{P}^1$ whose fibers are
infinite coincide with the image $L$ of the fundamental set $P$, and
$L$ is the intersection of $Q$ and the branched set $B$ ,
\vspace{.1in}

$\begin{array}{lcl} \chi(s)& = &
2\chi(\mathbb{P}^1\times\mathbb{P}^1 - B - Q) + \chi(Q + B - L) +
\chi(\phi^{-1}(L)) \\ \\
& = & 2\chi(\mathbb{P}^1\times\mathbb{P}^1) - \chi(Q) -
\chi(B)-\chi(L) + \chi(\phi^{-1}(L)) \\
\end{array}$
\vspace{.1in}

\noindent The Euler characteristic of
$\mathbb{P}^1\times\mathbb{P}^1$ is
$\chi(\mathbb{P}^1\times\mathbb{P}^1) = 4$.   As $Q$ is the disjoint
union of three twice-punctured spheres, $\chi(Q)=3(\chi(\mathbb{P}^1
)-2)=0$.  Since $B$ is three $\mathbb{P}^1$ 's which intersect at
four points, $\chi(B)=3\chi(\mathbb{P}^1) - 4\chi(point) = 2 $. Now
$L$ is just six points so $\chi(L)= 6$. That $\phi^{-1}(L)$ is the
union of six $\mathbb{P}^1$'s which intersect in pairs at a point
implies that $\chi(\phi^{-1}(L))= 6\chi(\mathbb{P}^1) -
3\chi(points)= 9$. All together this gives $\chi(S) = 9 $.

\end{proof}

\begin{cor}\label{cor:chi Ssmooth} The Euler characteristic of $\tilde{S}$ is $\chi(\tilde{S}) = 13$ \end{cor}
\begin{proof}
We have $ \chi(\tilde{S})  =  \chi(S) + 4 = 9+4 = 13$.
\end{proof}

\vspace{.2in}\noindent We are now ready to prove Theorem
\ref{thm:main}.

\vspace{.1in}\begin{proof}{(Theorem \ref{thm:main})} It follows from
lemma \ref{lem:birationalconicbundle} and corollary
\ref{prop:blowup}
 that $\chi(\tilde{S})=\chi(\mathbb{P}^2) + n$.  By corollary \ref{cor:chi
Ssmooth}, $n$ must be $10$ and $\tilde{S}$ must be $\mathbb{P}^2$
blown-up at $10$ points.
\end{proof}

%% file: asec_examples_short.tex
The Whitehead link complement can be obtained by $1/1$ Dehn surgery
on the Borromean rings (the complement of which we will denote by
$M_{br}$).

\begin{figure}[!hbp]
\begin{center}
\includegraphics[scale=.4]{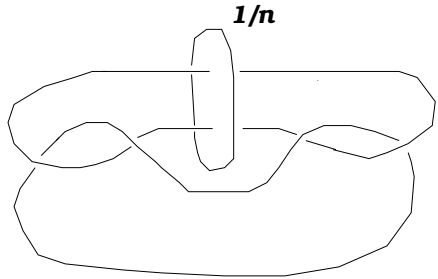}
\caption{$1/n$ Dehn surgery on the Borromean rings.\label{pic:br}}
\end{center}
\end{figure}

\vspace{.15in}
\noindent The manifolds which results from $1/n$ Dehn filling on one of
the cusps of $M_{br}$ are (\cite{HS}) two component 2-bridge link with Schubert normal form $S(8n,4n+1)$.
The fundamental group of these two component 2-bridge links has a
presentation of the form $\Gamma = \langle a,b| aw=wa \rangle$ with
$w = b^{\epsilon_1}a^{\epsilon_2}\dots b^{\epsilon_{8n-1}}$ where
$\epsilon_i = (-1)^{\lfloor \frac{i(4n-1)}{8n}\rfloor }$ .  For
$n=1,\dots,4$ we were able to use mathematica to determine the
polynomials which define the character varieties of $M_{br}(1/n)$.
For $n \ge 4$ the polynomials are a bit too large for mathematica to
handle.

Although we have few examples, there are some trends among these
character varieties which make them worth noting.  We can look at
the number of components and the number of canonical components
which comprise these character varieties.  Although the defining polynomials for certain components may not be smooth,
their bidgeree and how they change with the surgery coefficient $n$ is still of interest.  Below we summarize
this information.

\vspace{.2in}

\begin{center}
\begin{table}[!hbp]\label{table:example} Character Varieties for $M_{br}(1/n)$ \\
\begin{tabular}{|r||c|c|c|} \hline
&&& \\

Manifold & component & canonical component& bidegree \\

&&& \\ \hline \hline

$M_{br}(1/1) $ & 1 & $\surd$ & (2,3)  \\ \hline \hline

$M_{br}(1/2) $ & 1 &   & (2,2)  \\ \cline{2-4}
               & 2 &  $\surd$ & (4,5)  \\ \hline \hline

$M_{br}(1/3) $& 1 &   & (2,2)  \\ \cline{2-4}
              & 2 &  & (2,2)  \\ \cline{2-4}
              & 3 &  $\surd$ & (6,7)  \\ \hline \hline

$M_{br}(1/4) $& 1 &   & (2,2)  \\ \cline{2-4}
              & 2 &   & (4,4)  \\ \cline{2-4}
              & 3 &  $\surd$ & (8,9)  \\ \hline

\end{tabular}
\end{table}

\end{center}
 \vspace{.2in}

Possibly the most interesting characteristic these character
varieties share is the existence of a component which is defined by
a polynomial of bidegree $(2,k)$ where $k = \lbrace2,3\rbrace$.  All
of these components are $\mathbb{P}^1$ bundles over $\mathbb{P}^1$.
Although they are not conic bundles due to the existence of
singularities they all share a common feature.  Over the same
$\mathbb{P}^1$ coordinate ($[z,w]=[1,0]$), they all have a double
line fiber. See figure \ref{pic:1nconic}.

\begin{figure}[!hbp]
\begin{center}
\includegraphics[scale=.4]{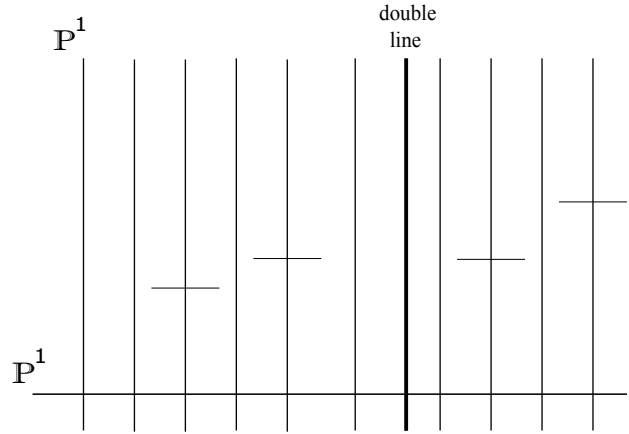}
\caption{component birational to a conic bundle for $M_{br}(1/n)$\label{pic:1nconic}}
\end{center}
\end{figure}

\noindent Away from a few points all of these components look like
conic bundles. All conics are parameterized by $\mathbb{P}^5$ and so
we can think of conic bundles as curves in $\mathbb{P}^5$. All the
degeneracies live in a hypersurface in $\mathbb{P}^5$ and all the
double lines live in a codimension two subvariety inside this
hypersurface. Hence, it is fairly uncommon for a curve in
$\mathbb{P}^5$ to intersect the subvariety which corresponds to
double line fibers.

All of these components are defined by polynomials that have
singularities (four for the Whitehead link and two for each of the
other Dehn surgery compoents).  While these $\mathbb{P}^1$ bundles
are not isomorphic to conic bundles, they are birational to such.
Since surfaces birational to conic bundles are rational, all of
these components are rational surfaces and thus isomorphic either to
$\mathbb{P}^1\times\mathbb{P}^1$ or $\mathbb{P}^2$ blown-up at some
number of points. As we did for the Whitehead link complement, we
can use the Euler characteristic to determine these character
varieties components topologically.  Aside from the Whitehead link
all of these components of these character varieties are
hypersurfaces defined by a singular polynomial of bidegree $(2,2)$
in $\mathbb{P}^2\times\mathbb{P}^1$. Each of these defining
polynomials has two singularities which each resolve into a conic
after a single blow-up.  Hence the Euler characteristic of the
smooth models are equal to that of the singular models plus two. The
way we calculate the Euler characteristic of these singular models
is very similar to  way to we calculated such for the Whitehead
link. It turns out that all of these singular (and so smooth) models
have Euler characteristic 10 and hence are all isomorphic to
$\mathbb{P}^2$ blown-up at 7 points.  From an algebro-geometric
perspective $\mathbb{P}^2$ blown-up at $7$ points is interesting in
the sense that is has only finitely many (precisely 49)  $(-1)$
curves.

What we have just described is a brief outline of the proof of
Theorem \ref{thm:1/n components}.

\begin{thm} For $n=2,\dots,4$, the character variety of $M_{br}(1/n)$
 has a component which is a rational surface isomorphic to $\mathbb{P}^2$ blown-up at 7 points.
\end{thm}

\noindent The proof is very similar to that for \ref{thm:wh}.  Hence
we omit the details here.  Below we have listed the singular
defining polynomials for these conic bundle components.

\vspace{.1in}
\begin{center}

\begin{table}[!hbp]\label{table:conics} Conic bundle components for $M_{br}(1/n)$ \\
\begin{tabular}{|r|c|c|c|} \hline
&&& \\

Manifold & singular defining polynomial & Euler  & Smooth Complex \\
        &  for conic bundle component & characteristic &  Surface\\

&&& \\ \hline \hline

$M_{br}(1/1) $ & $-w^3xy + w^2x^2z + w^2y^2z - wxyz^2 + u^2(z^3 -
2w^2z)$ & 13 & $\mathbb{P}^2$ blown-up at 10 points
\\ \hline

$M_{br}(1/2) $ & $ w^2x^2 + w^2y^2 - wxyz + u^2(z^2-2w^2) $  & 10 & $\mathbb{P}^2$ blown-up at 7 points \\
\hline

$M_{br}(1/3) $& $ w^2x^2 + w^2y^2 - wxyz + u^2(z^2-3w^2) $ &  10 & $\mathbb{P}^2$ blown-up at 7 points  \\
\cline{2-4}
              & $ w^2x^2 + w^2y^2 - wxyz + u^2(z^2-w^2) $ &  10 & $\mathbb{P}^2$ blown-up at 7 points  \\ \hline

$M_{br}(1/4) $& $w^2x^2 + w^2y^2 - wxyz + u^2(z^2-2w^2)$  &  10 & $\mathbb{P}^2$ blown-up at 7 points \\
\hline

\end{tabular}
\end{table}
\end{center}